\documentclass{amsart}
\usepackage{amssymb}
\newtheorem{Theo}{Theorem}
\newtheorem{prop}{Proposition}

\begin{document}
\newcommand{\Irr}{\text{Irr}}
\title[Frobenius' diophantine problem]{An estimate for Frobenius'
  diophantine problem in three dimensions}
\author[J.-C. Schlage-Puchta]{Jan-Christoph Schlage-Puchta}
\begin{abstract}
We give upper and lower bounds for the largest integer not representable
as positive linear combination of three given integers, disproving an
upper bound conjectured by Beck, Einstein and Zacks.
\end{abstract}
\maketitle


MSC-index: 11D04

Let $a, b, c$ be integers. The Frobenius problem is to determine the
greatest integer $n$ which cannot be represented as positive linear
combination of $a, b, c$; this integer will henceforth be denoted by
$f(a, b, c)$. Of course, if $a, b, c$ have a common divisor $m\geq 2$,
an integer not divisible by $m$ cannot be represented as integral linear
combination of $a$, $b$, and $c$, we will therefore suppose that $(a, b,
c)=1$. 

Beck, Einstein and Zacks\cite{BEZ} conjectured that, apart from
certain explicit families of exceptions, we have $f(a, b, c)\leq
(abc)^{5/8}$. They supported their conjuecture by an impressive amount
of numerical data and stated that counterexamples are not to be expected
unless $a, b, c$ are close to some arithmetic progression. However, in
this note we will show that their 
conjecture fails for somewhat generic examples, while it holds true for
the vast majority of triples. We do so by using diophantine
approximations inspired by the circle method. The idea to use such an
approach is already present in the work of Beck, Diaz and
Robins\cite{BDR}, however, it has not yet been fully exploited.
Define $N_a(b, c)$ to be the least integer
$n$ such that for every integer $x$ the congruence $x\equiv \nu a + \mu
b\pmod{a}$ is solvable with $0\leq\nu, \mu\leq n$. The basic relations
between the functions $f$ and $N$ are summarized as follows.
\begin{prop}
Let $a, b, c$ be positive integers satisfying $(a, b, c)=1$.
\begin{enumerate}
\item We have
\[
\min(b, c) N_a(b, c)\leq f(a, b, c)\leq (b+c) N_a(b, c).
\]
\item Suppose that $(a, b)=1$. Denote by $\overline{b}$ the modular
  inverse of $b$ modulo $a$. Then we have $N_a(b, c)=N_a(1,
  c\overline{b})$.
\item Suppose that $(p, q)=1$ and 
\[
\left|\left\{\frac{c\overline{b}}{a}\right\}-\frac{p}{q}\right| = \delta
\leq\frac{1}{q^2}.
\]
Then
\[
\min\big(\frac{a}{q}, \frac{1}{\delta q}\big) \ll N_a(b, c) \ll
\frac{a}{q} + q.
\]
\end{enumerate}
\end{prop}
\begin{proof}
(i) In every residue class modulo $a$ there exists an element $x$ which can
be represented as $x=b\nu+c\mu$ with $0\leq \nu, \mu\leq N_a(b, c)$, in
particular, $x\leq (b+c) N_a(b, c)$. Hence, every integer $n\geq
(b+c)N_a(b, c)$ can be written as $n=a\alpha+b\nu+c\mu$ with $\alpha\geq
0$ and $0\leq \nu, \mu\leq N_a(b, c)$, thus, $f(a, b, c)\leq  (b+c)
N_a(b, c)$. Conversely, there exists an integer $x$ such that the
congruence $x\equiv \nu a + \mu b\pmod{a}$ is unsolvable with $0\leq\nu,
\mu\leq N_a(b, c)-1$, that is, every element in this class which is
representable by $a, b, c$ can only be represented using $N_a(b, c)$
summands $b$ or $c$, and is therefore of size at least $\min(b, c)N_a(b,
c)$, that is, $f(a, b, c)\geq \min(b, c)N_a(b, c)$.

(ii) Our claim follows immediately from the fact that $x\equiv
b\nu+c\mu\pmod{a}$ is equivalent to $x\overline{b}\equiv
\nu+c\overline{b}\mu\pmod{a}$.

(iii) Define the integers $0=x_0<x_1<\dots<x_{q-1}<a$ by the relation
$\{x_0, \ldots, x_{q-1}\} =  \{0, c\overline{b}\bmod a,
2c\overline{b}\bmod a, \ldots, (q-1)c\overline{b}\bmod a\}$. Then we have
\[
\max (x_{i+1}-x_i) \leq \frac{a}{q} + (q-1)\delta \leq \frac{a}{q}+q,
\]
hence, every residue class $x$ modulo $a$ can be written as
$\nu+\mu c\overline{b}$ with $\nu\leq  \frac{a}{q}+q$ and $\mu\leq q-1$,
hence, $N_a(1, c\overline{b})\leq \frac{a}{q}+q$. Together with (ii),
the upper bound follows. For the lower bound suppose without loss that
$\left\{\frac{c\overline{b}}{a}\right\}>\frac{p}{q}$, and consider
$x=\lfloor\frac{a}{q}-1\rfloor$. We claim that $x$ cannot be represented
with less than $\frac{1}{2}\min\big(\frac{a}{q}, \frac{1}{\delta q}\big)$ summands
modulo $a$. Suppose first that $\delta<\frac{1}{a}$. Then $1\leq
qc\overline{b}\bmod a \leq q-1$, and no multiple
$\mu c\overline{b}\bmod{a}$ with $q\nmid \mu$, $\mu\leq\frac{a}{q}$
falls in the range $[1, x]$. Hence, for all pairs $\nu, \mu$ with
$1\leq\nu, \mu\leq \frac{a}{2q}$, such that
$\nu+\mu c\overline{b}\bmod{a}\in[1, x]$ we have $q|\mu$ and
\[
\nu+\mu c\overline{b}\bmod{a} \leq \nu + \frac{q-1}{q}\mu < \lfloor\frac{a}{q}\rfloor,
\]
which implies our claim in this case. If $\delta>\frac{1}{a}$, and
$\mu<\frac{1}{2\delta q}$, $\mu c\overline{b}\bmod{a}$ either does not
fall into the range $[1, x]$, or satisfies $\mu
c\overline{b}\bmod{a}\leq \delta\mu$, which implies the lower bound in
this case as well.
\end{proof}
From this result we draw the following conclusion:
\begin{Theo}
For each integer $a$ there are $\frac{\varphi(a)}{2}$ pairs $(b, c)$ such that
$a<b<c<2a$, $(a, b)=(a, c)=1$, and $\frac{a^2}{2}\leq f(a, b, c)\leq 2a^2$. Moreover, for
$\epsilon>0$, $\alpha, \beta\in[0, 1]$ and $a>a_0(\epsilon)$, there exist $b
\in[(1+\alpha)a, (1+\alpha+\epsilon) a]$ and $c\in[(1+\beta)a,
(1+\beta+\epsilon) a]$, such that $f(a, b, c)\geq\frac{a^2}{2}$. On the
other hand, the number of pairs $(b, c)$ such that $a<b<c<2a$, $(a, b, c)=1$ and
$f(a,b,c)>a^{3/2+\delta}$ is bounded above by $\mathcal{O}(a^{2-2\delta})$.
\end{Theo}
\begin{proof}
Choose $b\in[a, 2a]$ subject to the condition $(a, b)=1$, then choose
$c\in [a, 2a]$ such that $bc\equiv 1\pmod{a}$. The proposition implies that
$N_a(b, c)\geq\frac{a}{2}$, whereas the upper bound follows from $f(a,
b, c)\leq f(a, b)\leq 2a^2$. Since there are $\varphi(a)$ choices for
$b$, and at least half of them satisfy $b<\overline{b}$, our first claim
follows. For the second claim it suffices to prove that for every
$\alpha, \beta\in[0, 1]$ there exist integers $b\in[(1+\alpha)a,
(1+\alpha+\epsilon) a]$ and $c\in[(1+\beta)a, (1+\beta+\epsilon) a]$,
such that $bc\equiv 1\pmod{a}$. However, this follows immediatelly from
Weil's estimate for Kloosterman sums\cite{Weil} together with the
Erd\H os-Tur\'an-Koksma inequality (cf., e.g., \cite[p. 116]{KN}).
Finally, let $\delta>0$ be fixed, and set
$Q=a^{1/2+\delta}$. For every real number $\alpha\in[0, 1]$ there exists
some $q\leq Q$, such that $|\alpha-\frac{p}{q}|\leq\frac{1}{qQ}$. Denote
by $\mathfrak{M}$ the set of all $\alpha\in[0, 1]$, such that
$|\alpha-\frac{p}{q}|<\frac{1}{q^2}$ does not hold for any
$q\in[a^{1/2-\delta}, a^{1/2+\delta}]$. Then $\mathfrak{M}$ consists of
$a^{1-2\delta}$ intervals of total measure bounded above by
\[
\sum_{q\leq a^{1/2-\delta}} \frac{\varphi(q)}{qQ} \leq a^{-2\delta},
\]
thus, there are $\mathcal{O}(a^{1-2\delta})$ integers $\nu$, such that
$N_a(1, \nu)\geq a^{1/2+\delta}$, and our claim follows.
\end{proof}

Jan-Christoph Schlage-Puchta\\
Mathematisches Institut\\
Eckerstr. 1\\
79104 Freiburg\\
Germany\\
jcsp@mathematik.uni-freiburg.de
\end{document}